\documentclass[12pt]{amsart}

\usepackage{amscd,mathrsfs,epsfig,amsthm,amssymb,amsmath}
\usepackage[all]{xy} %\UseAllTwocells \SilentMatrices
\usepackage[T1]{CJKutf8}
\usepackage[sans]{dsfont}  % \mathds
\usepackage{hyperref}
\usepackage{tabularx}
\usepackage{rotating}

\usepackage{graphicx} % Required for inserting images
\usepackage{cite}
\usepackage{url}

\newtheorem{theorem}{Theorem}[subsection]

\newtheorem{definition}[theorem]{Definition}

\newtheorem{remark}[theorem]{Remark}

\newtheorem{example}[theorem]{Example}

\newtheorem*{thma}{Theorem A}

\allowdisplaybreaks[2]

\def\calC{{\mathcal C}}

\def\calJ{{\mathcal J}}

\def\frakM{{\mathfrak M}}

\def\frakF{{\mathfrak F}}

\def\frakR{{\mathfrak R}}

\def\Hom{\mathop{\rm Hom}\nolimits}

\def\lim{\mathop{\varinjlim}\nolimits}
 
\def\Ob{\mathop{\rm Ob}\nolimits}

\DeclareMathOperator{\C}{\mathbf{C}}

\DeclareMathOperator{\Ab}{\rm Ab}
\DeclareMathOperator{\Ring}{\rm Ring}
\DeclareMathOperator{\rMod}{Mod-}

\DeclareMathOperator{\Add}{\rm Add}
\DeclareMathOperator{\op}{\rm op}

\begin{document}

\title[modules over a sum-id bipresheaf of rings]{Non-abelianess of the category of modules over a sum-id bipresheaf of rings}

\author{Mawei Wu}
%\author{Fei Xu}
\address{School of Mathematics and Statistics, Lingnan Normal University, Zhanjiang, Guangdong 524048, China}
%\address{Department of Mathematics\\Shantou University\\Shantou, Guangdong 515063, China}
\email{wumawei@lingnan.edu.cn}
%\email{19mwwu@stu.edu.cn}
%\email{fxu@stu.edu.cn}



%\thanks{The authors \begin{CJK*}{UTF8}{}
%\CJKtilde \CJKfamily{gbsn}(吴马威、徐斐)
%\end{CJK*} are supported by the NSFC grants No.12171297 and No.11671245}

%\date{March 21, 2024}
%\date{\today}

\begin{abstract}
Let $\calC$ be a small category, motivated by the definition of bisheaves of abelian groups of MacPherson and Patel (see \cite[Definition 5.1]{MP21}), we first introduce the notions of bipresheaves of rings $\frakR$ on $\calC$ and their module categories $\rMod \frakR$. Then the linear Grothendieck construction $Gr(\frakR)$ of $\frakR$ is defined. With this linear Grothendieck construction, we show that the category of bipresheaves of modules over a sum-id bipresheaf of rings $\frakR$ can be characterized as the category of bipresheaves of abelian groups on $Gr(\frakR)$. It follows that the category $\rMod \frakR$ of modules over a sum-id bipresheaf of rings $\frakR$ is non-abelian.  
\end{abstract}

\maketitle

\tableofcontents

\section{Introduction}
Let $\C=(\calC, \calJ)$ be a small site and let $R$ be a sheaf of rings on $\calC$, Howe in 1981 showed that the category of $R$-modules is equivalent to the category of sheaves of abelian groups on $Gr(R)$ (see \cite[Proposition 5]{How81}), where the category $Gr(R)$, which is called by Howe the semi-direct product of $\calC$ by $R$, is actually the linear Grothendieck construction of $R$ (see \cite[Definition 4.1]{Asa13b}). Instead of considering a sheaf of rings on $\calC$ and its modules, in 2024, we investigated a representation of the category $\calC$ (a pseudofunctor from a small category $\calC$ to the category of small preadditive categories $\Add$) and a dg-representation of the category $\calC$ (a pseudofunctor from a small category $\calC$ to the category of small dg $k$-categories $\mbox{dg-Cat}_k$, where $k$ is a commutative unital ring), as well as their modules. Using the linear Grothendieck construction, we obtained some category equivalences analogous to that of Howe, see \cite[Theorem A]{Wu24pseudo} and \cite[Theorem A]{Wu24dg}. 

In this paper, motivated by the definition of bisheaves of abelian groups of MacPherson and Patel (see \cite[Definition 5.1]{MP21}), we first introduce the notions of bipresheaves $\frakR$ of rings on $\calC$ (see Definition \ref{bishring}) and their modules (see Definition \ref{bishmod}). Our main goal is to try to extend Howe's equivalence (see \cite[Proposition 5]{How81}) to the setting of bipresheaves. For our purpose, we first give the definition of the linear Grothendieck construction $Gr(\frakR)$ of $\frakR$ (see Definition \ref{grocon}). With this linear Grothendieck construction, we show that the category of bipresheaves of modules over a sum-id bipresheaf $\frakR$ of rings (a particular kind of bipresheaf of rings, see Definition \ref{specialbishring}) can be characterized as the category $\rMod Gr(\frakR)$ of bipresheaves of abelian groups on $Gr(\frakR)$. It follows that the category $\rMod \frakR$ of $\frakR$-modules is non-abelian. To be more precise, we have the following theorem.

\begin{thma} {\rm (Theorem \ref{mainthm})}
Let $\calC$ be a small category and let $\frakR$ be a sum-id bipresheaf of rings on $\calC$, then we have the following category equivalence
 $$
 \rMod \frakR \simeq \rMod Gr(\frakR).
 $$ 
Consequently, the category $\rMod \frakR$ of bipresheaves of modules over $\frakR$ is non-abelian.     
\end{thma}

This paper is organized as follows. In Section \ref{bish}, we first recall the definitions of (co)presheaves of rings and their modules, and then some new notions, bipresheaves of rings and their modules, are introduced. In Section \ref{cha}, the definitions of the linear Grothendieck construction of a bipresheaf of rings and a special kind of bipresheaf of rings, namely, the sum-id bipresheaf of rings, are introduced. Then, using the linear Grothendieck construction, a characterization of the category of modules over a sum-id bipresheaf of rings is given.

Throughout this paper, we will denote by $\Ring$ the category of unital commutative rings and unital ring homomorphisms. The category of abelian groups is denoted by $\Ab$. All the Grothendieck topologies considered here are trivial topologies, that is, we mainly focus on studying (co)presheaves rather than (co)sheaves. All rings considered in this paper are assumed to be unital commutative, and all ring homomorphisms are assumed to be unital.

\section{Bipresheaves of rings and modules} \label{bish}
In this section, we will first recall the definitions of (co)presheaves of rings and their modules, and then bipresheaves of rings and their modules will be introduced. 
\subsection{(Co)presheaves of rings and modules}
Let $\calC$ be a small category, in this subsection, the definitions of (co)presheaves of rings on $\calC$ and their modules will be recalled. We refer the reader to \cite{stackpro,KS06} for more information. 
\begin{definition}
Let $\calC$ be a small category. A \emph{presheaf of rings on} $\calC$ is a contravariant functor from $\calC$ to the category of rings.  
\end{definition}

\begin{definition}
Let $\calC$ be a small category. A \emph{copresheaf of rings on} $\calC$ is a covariant functor from $\calC$ to the category of rings.   
\end{definition}

\begin{definition} {\rm (\cite[Definition 18.9.1]{stackpro})}
Let $\calC$ be a category, and let $\frakR_1: \calC^{\op} \to \Ring$ be a presheaf of rings on $\calC$. A \emph{presheaf of $\frakR_1$-modules} is given by an abelian presheaf $\frakM_1$ together with a map of presheaves of sets
    $$
    \frakM_1 \times \frakR_1 \to \frakM_1
    $$
    such that for every object $x$ of $\calC$ the map $\frakM_1(x) \times \frakR_1(x) \to \frakM_1(x)$ defines the structure of an $\frakR_1(x)$-module structure on the abelian group $\frakM_1(x)$.
\end{definition}

Dually, we also have the notion of a \emph{copresheaf of $\frakR_2$-modules} $\frakM_2$ for a copresheaf of rings $\frakR_2$ on $\calC$.

\subsection{Bipresheaves of rings and modules}
In this subsection, we will introduce some new notions, that is, bipresheaves of rings and their modules. Let us first recall the definition of biprsheaves of abelian groups. 

\begin{definition} {\rm (\cite[Definition 5.1]{MP21})} \label{bishab}
Let $\calC$ be a small category. A \emph{bipresheaf of abelian groups on} $\calC$ is a triple $\mathfrak{A}:=(\mathfrak{A}_1, \mathfrak{A}_2, \eta)$ where $\mathfrak{A}_1: \calC^{\op} \to \Ab$ is a presheaf of abelian groups on $\calC$, $\mathfrak{A}_2: \calC \to \Ab$ is a copresheaf of abelian groups on $\calC$, and $\eta:=\{\eta_x: \mathfrak{A}_1(x) \to \mathfrak{A}_2(x) ~|~ x \in \Ob \calC \}$ is a set of group homomorphisms such that, for each morphism $f: x \to y$ in $\calC$, the following diagram commutes:  
     $$
     \xymatrix{
      \mathfrak{A}_1(x) \ar[d]_{\eta_x} & \mathfrak{A}_1(y) \ar[l]_{\mathfrak{A}_1(f)} \ar[d]^{\eta_y} \\
      \mathfrak{A}_2(x) \ar[r]^{\mathfrak{A}_2(f)} & \mathfrak{A}_2(y), \\
     } 
     $$ 
that is, we have
$$
\eta_y=\mathfrak{A}_2(f) \circ \eta_x \circ \mathfrak{A}_1(f).
$$
\end{definition}

When replacing abelian groups with rings and group homomorphisms with unital ring homomorphisms, one can define the notion of bipresheaves of rings on $\calC$.

\begin{definition} \label{bishring}
Let $\calC$ be a small category. A \emph{bipresheaf of rings on} $\calC$ is a triple $\frakR:=(\mathfrak{R}_1, \mathfrak{R}_2, \theta)$ where $\frakR_1: \calC^{\op} \to \Ring$ is a presheaf of rings on $\calC$, $\frakR_2: \calC \to \Ring$ is a copresheaf of rings on $\calC$, and $\theta:=\{\theta_x: \frakR_1(x) \to \frakR_2(x) ~|~ x \in \Ob \calC \}$ is a set of unital ring homomorphisms such that, for each morphism $f: x \to y$ in $\calC$, the following diagram commutes:  
     $$
     \xymatrix{
      \frakR_1(x) \ar[d]_{\theta_x} & \frakR_1(y) \ar[l]_{\frakR_1(f)} \ar[d]^{\theta_y} \\
      \frakR_2(x) \ar[r]^{\frakR_2(f)} & \frakR_2(y), \\
     } 
     $$ 
that is, we have
$$
\theta_y=\frakR_2(f) \circ \theta_x \circ \frakR_1(f).
$$
\end{definition}

Given a bipresheaf of rings, the definition of its modules is given as follows. 

\begin{definition} \label{bishmod}
Let $\calC$ be a small category and let $\frakR:=(\mathfrak{R}_1, \mathfrak{R}_2, \theta)$ be a bipresheaf of rings on $\calC$. A \emph{bipresheaf of modules over} $\frakR$ is a bipresheaf of abelian groups $\frakM:=(\mathfrak{M}_1, \mathfrak{M}_2, \eta)$ on $\calC$ such that $\frakM_1$ is a presheaf of $\frakR_1$-modules, $\frakM_2$ is a copresheaf of $\frakR_2$-modules, and for each $x \in \Ob \calC, r \in \frakR_1(x)$, $m \in \frakM_1(x)$,
$$
\eta_x( m\cdot_1 r) = \theta_x(r) \cdot_2 \eta_x(m),
$$
where $\cdot_i$ mean $\frakR_i(x)$-actions on $\frakM_i(x)$ for $i=1,2$ repectively. We will omit the subscripts when there is no ambiguity.  
\end{definition}

Let $\calC$ be a small category and let $\frakR$ be a bipresheaf of rings on $\calC$, we will denote \emph{the category of $\frakR$-modules} and \emph{the category of bipresheaves of abelian groups on} $\calC$ by $\rMod \frakR$ and $\rMod \calC$, respectively.

\section{A characterization of the category of modules over a sum-id bipresheaf of rings} \label{cha}
In this section, we will first give the definitions of the linear Grothendieck construction of a bipresheaf of rings and a particular type of bipresheaf of rings, so-called the sum-id bipresheaf of rings. Then, using this linear Grothendieck construction, a characterization of the category of modules over a sum-id bipresheaf of rings will be given.

\subsection{Linear Grothendieck constructions of bipresheaves of rings}
In this subsection, a new construction, the linear Grothendieck construction of a bipresheaf of rings will be given. We already knew the linear Grothendieck construction of a (co)presheaf of rings (see \cite[Definition 2.2.1]{Wu24torsion}), so the key to define the linear Grothendieck construction of a bipresheaf of rings lies in encoding the connecting morphisms $\theta_x$ in Definition \ref{bishring}. By the work of Mitchell \cite{Mit72}, we know that a unital ring can be viewed as a single-object category, we will denote by $\bullet$ for its object. And a unital ring homomorphism can be seen as a functor.

\begin{definition} \label{grocon}
 Let $\calC$ be a small category and let $\frakR:=(\mathfrak{R}_1, \mathfrak{R}_2, \theta)$ be a bipresheaf of rings on $\calC$. Then a category $Gr(\frakR)$, called \emph{the linear Grothendieck construction of} $\frakR$, is defined as follows:
 \begin{enumerate}
     \item objects:
    $$
    \Ob Gr(\frakR):= \{~(x, \bullet_{x_1}, \bullet_{x_2}) ~|~ x \in \Ob \calC, \bullet_{x_1} \in \Ob \frakR_1(x), \bullet_{x_2} \in \Ob \frakR_2(x)~\} \overset{iden.~with}{=}\Ob \calC; 
    $$
     \item morphisms: for each $x, y \in \Ob Gr(\frakR)$, 
     let 
     $$
     \frakR_1(x) \otimes_{\theta} \frakR_2(y) := \{~ r_1 \otimes \frakR_2(f)(\theta_x(r_1))r_2 ~|~ f \in \Hom_{\calC}(x,y), r_1 \in \frakR_1(x), r_2 \in \frakR_2(y)~\}, 
     $$
     where the map $\frakR_2(f)(\theta_x(-))$ can be depicted as follows:
    $$
    \xymatrix @C=3pc {
    \frakR_1(x) \ar[r]^{\theta_x} & \frakR_2(x) \ar[r]^{\frakR_2(f)} &  \frakR_2(y),\\
     } 
    $$
     for simplicity, we will write $r_1 \otimes \frakR_2(f)(\theta_x(r_1))r_2$
     as $r_1 \otimes r_1^{\theta} r_2$.
     We set
     $$
     \Hom_{Gr(\frakR)}(x, y) := \bigoplus_{f \in \Hom_{\calC}(x,y)} \frakR_1(x) \otimes_{\theta} \frakR_2(y); 
     $$
     
     \item composition: for each
     $$
     \left(\cdots, \left(r_1 \otimes \frakR_2(f)(\theta_x(r_1))r_2 \right)_f, \cdots \right)_{f \in \Hom_{\calC}(x,y)} \in \Hom_{Gr(\frakR)}(x, y)
     $$
     and 
     $$
     \left(\cdots, \left(s_1 \otimes \frakR_2(g)(\theta_y(s_1))s_2 \right)_g, \cdots \right)_{g \in \Hom_{\calC}(y,z)} \in \Hom_{Gr(\frakR)}(y, z),
     $$
     the composition of them is given as follows: let $r_1^{\theta}:=\frakR_2(f)(\theta_x(r_1))$ and $s_1^{\theta}:=\frakR_2(g)(\theta_y(s_1))$, then    
     \begin{align*}
     & \left(\cdots, \left(s_1 \otimes s_1^{\theta} s_2 \right)_g, \cdots \right)_{g \in \Hom_{\calC}(y,z)} \circ 
     \left(\cdots, \left(r_1 \otimes r_1^{\theta} r_2 \right)_f, \cdots \right)_{f \in \Hom_{\calC}(x,y)} \\
    :=  & \left(\cdots, \left(\sum_{h=gf} \frakR_1(f)(s_1)r_1 \otimes \frakR_2(g)\left(r_1^{\theta} r_2\right) \left(s_1^{\theta}s_2\right)\right)_h, \cdots \right)_{h \in \Hom_{\calC}(x,z).} \\
    \end{align*}     
 \end{enumerate}
\end{definition}

\begin{remark} \label{remark}
\begin{enumerate}
    \item the composition rule in the definition above is as follows: the first component $(\frakR_1(f)(s_1)r_1)gf$ of the tensor product is the morphism composition of $r_1f$ and $s_1g$ in the linear Grothendieck construction $Gr(\frakR_1)$ of the presheaf $\frakR_1$ (see \cite[the paragraph below the Definition 3.1.1]{WX23}), and the second component $\left( \left(s_1^{\theta}s_2\right) \frakR_2(g)\left(r_1^{\theta} r_2\right) \right)gf$ of the tensor product (note that all rings considered in this paper are commutative) is the morphism composition of the morphisms $(r_1^{\theta} r_2)f$ and $(s_1^{\theta}s_2)g$ in the linear Grothendieck construction $Gr(\frakR_2)$ of the copresheaf $\frakR_2$;
    \item the composition is well-defined: since 
    \begin{align*}
    & \frakR_2(g)\left(r_1^{\theta} r_2\right) \left(s_1^{\theta}s_2\right)\\
   [s_1^{\theta}:=\frakR_2(g)(\theta_y(s_1))] =~& \frakR_2(g)\left(r_1^{\theta} r_2\right) \left(\frakR_2(g)(\theta_y(s_1)) s_2\right) \\
   by~def. ~\ref{bishring} =~&\frakR_2(g)\left(r_1^{\theta} r_2\right) \left(\frakR_2(g)(\frakR_2(f) ( \theta_x ( \frakR_1(f) (s_1)))) s_2\right)\\
   [r_1^{\theta}:=\frakR_2(f)(\theta_x(r_1))] =~& \frakR_2(g)\left(\frakR_2(f)(\theta_x(r_1)) r_2\right) \left(\frakR_2(g)(\frakR_2(f) ( \theta_x ( \frakR_1(f) (s_1)))) s_2\right) \\
   \frakR_2(g)~is~a~funct. =~&\frakR_2(g)\left(\frakR_2(f)(\theta_x(r_1))\right) \frakR_2(g)\left(r_2\right) \left(\frakR_2(g)(\frakR_2(f) ( \theta_x ( \frakR_1(f) (s_1)))) s_2\right) \\
    \frakR_2~is~a~funct. =~& \frakR_2(gf)(\theta_x(r_1)) \frakR_2(g)\left(r_2\right) \left(\frakR_2(gf) ( \theta_x ( \frakR_1(f) (s_1))) s_2\right) \\
   rings~are~comm. =~& \frakR_2(gf)(\theta_x(r_1)) \frakR_2(gf) ( \theta_x ( \frakR_1(f) (s_1)))  \left(\frakR_2(g)\left(r_2\right) s_2\right) \\
   \frakR_2(gf) ~and~ \theta_x~are~funct. =~& \frakR_2(gf)\left(\theta_x \left( r_1 \frakR_1(f)(s_1) \right)\right) \left(\frakR_2(g)(r_2)s_2\right) \\
     rings~are~comm.=~& \frakR_2(gf)\left(\theta_x \left(\frakR_1(f)(s_1)r_1 \right)\right) \left(\frakR_2(g)(r_2)s_2\right) \\
    abbreviated~as =~ & \left(\frakR_1(f)(s_1)r_1 \right)^{\theta} \left(\frakR_2(g)(r_2)s_2\right), \\
    \end{align*}

(In order to better understand the equalities above, one can chase the following diagram for $r_1 \in \frakR_1(x)$, $r_2 \in \frakR_2(y)$, $s_1 \in \frakR_1(y)$, $s_2 \in \frakR_2(z)$ and $x \overset{f}{\to} y \overset{g}{\to} z$. The morphisms in the diagram below can be viewed both as functors and as unital ring homomorphisms, also all rings in the diagram below are assumed to be commutative.)
    $$
    \xymatrix @C=3pc { \ar @{} [dr]|{\circlearrowright}
    \frakR_2(x) \ar[r]^{\frakR_2(f)} & \ar @{} [dr]|{\circlearrowright} \frakR_2(y) \ar[r]^{\frakR_2(g)} &  \frakR_2(z)\\
    \frakR_1(x) \ar[u]^{\theta_x} & \frakR_1(y) \ar[u]_{\theta_y} \ar[l]^{\frakR_1(f)} &  \frakR_1(z) \ar[u]_{\theta_z} \ar[l]^{\frakR_1(g)} \\
     } 
    $$
    
 therefore, we have 
    \begin{align*}
     & \frakR_1(f)(s_1)r_1 \otimes \frakR_2(g)\left(r_1^{\theta} r_2\right) \left(s_1^{\theta}s_2\right) \\
     =~ & \frakR_1(f)(s_1)r_1 \otimes \left(\frakR_1(f)(s_1)r_1 \right)^{\theta} \left(\frakR_2(g)(r_2)s_2\right) \\
      \in~ & \frakR_1(x) \otimes_{\theta} \frakR_2(z). \\
    \end{align*}
It follows that 
$$
\left(\cdots, \left(\sum_{h=gf} \frakR_1(f)(s_1)r_1 \otimes \frakR_2(g)\left(r_1^{\theta} r_2\right) \left(s_1^{\theta}s_2\right)\right)_h, \cdots \right)_{h \in \Hom_{\calC}(x,z)} \in \Hom_{Gr(\frakR)}(x, z),
$$
so the composition is well-defined;  
    \item if $\mathds{1}_x := (\cdots, 0, \left(1_{\frakR_1(x)} \otimes 1_{\frakR_2(x)}\right)_{1_x \in \Hom_{\calC}(x,x)}, 0, \cdots)$, then $\mathds{1}_x \in \Hom_{Gr(\frakR)}(x, x)$, where $1_{\frakR_i(x)}$ (resp. $0$) is the identity element (resp. the zero element) of the ring $\frakR_i(x)$ for $i=1,2$. One can check that $\mathds{1}_x$ is the identity morphism on $x \in \Ob Gr(\frakR)$, and it is routine to check that $Gr(\frakR)$ is a category.     
\end{enumerate}    
\end{remark}

Now, a new class of bipresheaves of rings will be introduced, which will be used in the main characterization theorem later. 

\begin{definition} \label{specialbishring}
Let $\calC$ be a small category. A bipresheaf $\frakR$ of rings on $\calC$ 
is said to be \emph{sum-id} if, for each morphism of its linear Grothendieck construction 
     $$
     \left(\cdots, \left(r_1 \otimes \frakR_2(f)(\theta_x(r_1))r_2 \right)_f, \cdots \right)_{f \in \Hom_{\calC}(x,y)} \in \Hom_{Gr(\frakR)}(x, y),
     $$  
the following condition holds
$$
\sum_{\substack{f \in \Hom_{\calC}(x,y) \\ r_1 \in \frakR_1(x) \\ r_2 \in \frakR_2(y)}} r_1^{\theta} r_2=1. 
$$
That is, the (finite) \underline{sum} of all the second components of the tensor products of each morphism $ \left(\cdots, \left(r_1 \otimes r_1^{\theta} r_2 \right)_f, \cdots \right)_{f \in \Hom_{\calC}(x,y)}$ is equal to the \underline{id}entity. 
\end{definition}

\begin{example}
Let $\frakR=(\frakR_1,\frakR_2,\theta)$ be a bipresheaf of rings on $\calC$, where $\frakR_1: \calC^{\op} \to \Ring$ is an arbitrary presheaf of rings, $\frakR_2: \calC \to \Ring$ is a constant copresheaf valued in a trivial ring $\{1\}$ (i.e. a one-element ring with $0=1$), and $\theta_x: \frakR_1(x) \to \frakR_2(x)$ are trivial ring homomorphisms for all $x \in \Ob \calC$, then one can check that $\frakR$ is a sum-id bipresheaf of rings. 
\end{example}

\subsection{Characterization theorem}
In this subsection, we will characterize the category of modules of a sum-id bipresheaf of rings as the category of bipresheaves of abelian groups. The following theorem can be viewed as a bipresheaves version of Howe's characterization theorem (see \cite[Proposition 5]{How81}). 

\begin{theorem} \label{mainthm}
Let $\calC$ be a small category and let $\frakR$ be a sum-id bipresheaf of rings on $\calC$, then we have the following category equivalence
 $$
 \rMod \frakR \simeq \rMod Gr(\frakR).
 $$ 
Consequently, the category $\rMod \frakR$ of bipresheaves of modules over $\frakR$ is non-abelian.     
\end{theorem}

\begin{proof}
By \cite[the third line from the bottom of page 4]{MP21}, we know that the category of bisheaves of abelian groups is non-abelian. Hence, in order to prove the theorem, we only have to show that the category equivalence $\rMod \frakR \simeq \rMod Gr(\frakR)$ holds. Our strategy is to construct two functors which give rise to the desired equivalence. 

Firstly, let us define a functor 
    $$
    \Psi: \rMod \frakR \to \rMod Gr(\frakR)
    $$ 
by
   $$
	\mathfrak{M}=(\frakM_1, \frakM_2, \eta) \mapsto \mathfrak{F}_{\mathfrak{M}}=(\frakF_1, \frakF_2, \eta),
   $$ 
where $\frakF_1, \frakF_2$ are defined as follows:
$$
\xymatrix{
    Gr(\frakR)^{\op} \ar[rr]^{\frakF_1} &  & \Ab   \\
   x  \ar@{}[u]|{\begin{sideways}$\in$\end{sideways}} \ar[dd]_{\left(\cdots, \left(r_1 \otimes r_1^{\theta} r_2 \right)_f, \cdots \right)_f} & & \frakM_1(x) \ar@{}[u]|{\begin{turn}{90}$\in$\end{turn}}  \\
     & \longmapsto &  \\
   y  & & \frakM_1(y) \ar[uu]_{\sum_{r_1,f}(- \cdot r_1) \circ \frakM_1(f)}  \\
     } 
$$
and 
$$
\xymatrix{
    Gr(\frakR) \ar[rr]^{\frakF_2} &  & \Ab   \\
   x  \ar@{}[u]|{\begin{sideways}$\in$\end{sideways}} \ar[dd]_{\left(\cdots, \left(r_1 \otimes r_1^{\theta} r_2 \right)_f, \cdots \right)_f} & & \frakM_2(x) \ar@{}[u]|{\begin{turn}{90}$\in$\end{turn}} \ar[dd]^{\sum_{r_2,f}(r_2 \cdot -) \circ \frakM_2(f)} \\
     & \longmapsto &  \\
   y  & & \frakM_2(y)   \\
     } 
$$
It is not hard to check that $\frakF_1$ (resp. $\frakF_2$) is a presheaf (resp. copresheaf). We also have to check the following diagram commutes.
     $$
     \xymatrix @C=7pc {
      \frakM_1(x) \ar[d]_{\eta_x}  &  \frakM_1(y) \ar[d]^{\eta_y} \ar[l]_{\sum_{r_1,f}(- \cdot r_1) \circ \frakM_1(f)} \\
      \frakM_2(x) \ar[r]^{\sum_{r_2,f}(r_2 \cdot -) \circ \frakM_2(f)} &  \frakM_2(y) \\
     } 
     $$ 
This is true because, for each $m \in \frakM_1(y)$,
\begin{align*}
     &\left( \left(\sum_{r_2,f}(r_2 \cdot -) \circ \frakM_2(f) \right) \circ \eta_x \circ \left(\sum_{r_1,f}(- \cdot r_1) \circ \frakM_1(f)\right) \right)(m) \\
    =& \sum_{r_2,f} r_2 \cdot \frakM_2(f) \left( \eta_x \left( \sum_{r_1,f} \frakM_1(f)(m) \cdot r_1 \right) \right) \\
    \eta_x~is~a~ring~homo.=& \sum_{r_2,f} r_2 \cdot \frakM_2(f) \left(  \sum_{r_1,f} \left( \eta_x \left( \frakM_1(f)(m) \cdot r_1 \right) \right) \right) \\
    by~def.~\ref{bishmod}=& \sum_{r_2,f} r_2 \cdot \frakM_2(f) \left(\sum_{r_1,f} \theta_x(r_1) \cdot \eta_x( \frakM_1(f)(m) ) \right)  \\
    \frakM_2(f)~is~a~ring~homo.=& \sum_{r_2,f} r_2 \cdot \left(\sum_{r_1,f}  \frakM_2(f) \left(\theta_x(r_1) \cdot \eta_x( \frakM_1(f)(m) ) \right) \right)  \\
    \frakM_2~is~an~\frakR_2 \mbox{-} module=& \sum_{r_2,f} r_2 \cdot \left(\sum_{r_1,f} \left( \frakR_2(f)(\theta_x(r_1)) \cdot \frakM_2(f) \left(\eta_x( \frakM_1(f)(m) ) \right) \right) \right)  \\
    by~def.~\ref{bishab}= & \sum_{r_2,f} r_2 \cdot \left(\sum_{r_1,f} 
    \left( r_1^{\theta}  \cdot \eta_y(m) \right)  \right) \\
    =& \sum_{r_2,f} \sum_{r_1,f} \left( r_2 \cdot \left( r_1^{\theta}  \cdot \eta_y(m) \right) \right) \\
    =& \sum_{r_2,f} \sum_{r_1,f} \left( r_2 r_1^{\theta} \right) \cdot \eta_y(m)   \\
    rings~are~commu.=& \sum_{r_1,r_2,f} \left( r_1^{\theta} r_2 \right) \cdot \eta_y(m)  \\ 
    =& \left(\sum_{r_1,r_2,f} r_1^{\theta} r_2 \right) \cdot \eta_y(m)  \\
    by~def.~\ref{specialbishring}=& \eta_y(m). \\
\end{align*}

Secondly, let us define another functor 
    $$
    \Phi: \rMod Gr(\frakR) \to  \rMod \frakR
    $$ 
by
   $$
   \mathfrak{F}=(\frakF_1, \frakF_2, \eta) \mapsto  \mathfrak{M}_{\frakF}=(\frakM_1, \frakM_2, \eta),
   $$ 
where $\frakM_1, \frakM_2$ are defined as follows:
$$
\xymatrix{
    \calC^{\op} \ar[rr]^{\frakM_1} &  & \Ab   \\
   x  \ar@{}[u]|{\begin{sideways}$\in$\end{sideways}} \ar[dd]_{f} & & \frakF_1(x) \ar@{}[u]|{\begin{turn}{90}$\in$\end{turn}}  \\
     & \longmapsto &  \\
   y  & & \frakF_1(y) \ar[uu]_{\mathfrak{F}_1[(\cdots, 0, \left( 1_{\frakR_1(x)} \otimes 1_{\frakR_2(y)} \right)_f,0, \cdots)]}  \\
     } 
$$
and 
$$
\xymatrix{
    \calC \ar[rr]^{\frakM_2} &  & \Ab   \\
   x  \ar@{}[u]|{\begin{sideways}$\in$\end{sideways}} \ar[dd]_{f} & & \frakF_2(x) \ar@{}[u]|{\begin{turn}{90}$\in$\end{turn}} \ar[dd]^{\mathfrak{F}_2[(\cdots, 0, \left( 1_{\frakR_1(x)} \otimes 1_{\frakR_2(y)} \right)_f,0, \cdots)]} \\
     & \longmapsto &  \\
   y  & & \frakF_2(y)   \\
     } 
$$
It is not hard to check that $\frakM_1$ (resp. $\frakM_2$) is a presheaf (resp. copresheaf). We also have to check the following diagram commutes.
     $$
     \xymatrix @C=10pc {
      \frakF_1(x) \ar[d]_{\eta_x}  &  \frakF_1(y) \ar[d]^{\eta_y} \ar[l]_{\mathfrak{F}_1[(\cdots, 0, \left( 1_{\frakR_1(x)} \otimes 1_{\frakR_2(y)} \right)_f,0, \cdots)]} \\
      \frakF_2(x) \ar[r]^{\mathfrak{F}_2[(\cdots, 0, \left( 1_{\frakR_1(x)} \otimes 1_{\frakR_2(y)} \right)_f,0, \cdots)]} &  \frakF_2(y) \\
     } 
     $$ 
This is true as $\frakF$ is a bipresheaf of abelian groups, by Definition \ref{bishab}, the diagram above commutes.

We also have to define an $\frakR_1$(resp. $\frakR_2$)-module structure on $\frakM_1$(resp. $\frakM_2$). For an $\frakR_1$-module structure on $\frakM_1$, we define it as follows: for each $r_1 \in \frakR_1(x)$, $m \in \frakF_1(x)$,
$$
m \cdot r_1 := \frakF_1[(\cdots, 0, \left( r_1 \otimes r_1^{\theta} \right)_{1_x},0, \cdots)](m).
$$
This is indeed an action: because
\begin{align*}
  & m \cdot 1_{\frakR_1(x)} \\ 
  =& \frakF_1[(\cdots, 0, \left( 1_{\frakR_1(x)} \otimes 1_{\frakR_1(x)}^{\theta} \right)_{1_x},0, \cdots)](m) \\
  =& \frakF_1 (\mathds{1}_x) (m) \\
  =& m,
\end{align*}
and
\begin{align*}
  & (m \cdot r_1)  \cdot r_1' \\    
  =& \frakF_1[(\cdots, 0, \left( r_1' \otimes r_1'^{\theta} \right)_{1_x},0, \cdots)]  \left\{ \frakF_1[(\cdots, 0, \left( r_1 \otimes r_1^{\theta} \right)_{1_x},0, \cdots)](m) \right\}\\
 \frakF_1~is~a~presh.=& \frakF_1[(\cdots, 0, \left( r_1 \otimes r_1^{\theta} \right)_{1_x},0, \cdots) \circ (\cdots, 0, \left( r_1' \otimes r_1'^{\theta} \right)_{1_x},0, \cdots)]  (m) \\
 =& \frakF_1[(\cdots, 0, \left( r_1 r_1' \otimes r_1'^{\theta} r_1^{\theta} \right)_{1_x},0, \cdots)]  (m) \\
rings~are~comm. =& \frakF_1[(\cdots, 0, \left( r_1' r_1 \otimes r_1'^{\theta} r_1^{\theta} \right)_{1_x},0, \cdots)]  (m) \\
\theta_x~is~a~ring~homo.=& \frakF_1[(\cdots, 0, \left( r_1' r_1 \otimes (r_1' r_1)^{\theta} \right)_{1_x},0, \cdots)]  (m) \\
 =& m \cdot (r_1' r_1)  \\
 rings~are~comm.=& m \cdot (r_1 r_1').  \\
\end{align*}

Now, let us define an $\frakR_2$-module structure on $\frakM_2$: for each $s_1 \in \frakR_2(y)$, $m \in \frakF_2(y)$,
$$
s_1 \cdot m := \frakF_2[(\cdots, 0, \left( 1_{\frakR_1(y)} \otimes s_1 \right)_{1_y},0, \cdots)](m).
$$
This is indeed an action: because
\begin{align*}
    & 1_{\frakR_2(y)} \cdot m \\
    =& \frakF_2[(\cdots, 0, \left( 1_{\frakR_1(y)} \otimes 1_{\frakR_2(y)} \right)_{1_y},0, \cdots)](m)\\
    =& \frakF_2 (\mathds{1}_y) (m) \\
    =& m,
\end{align*}
and
\begin{align*}
    & s_1' \cdot ( s_1 \cdot m) \\
    =& \frakF_2[(\cdots, 0, \left( 1_{\frakR_1(y)} \otimes s_1' \right)_{1_y},0, \cdots)] \left\{ \frakF_2[(\cdots, 0, \left( 1_{\frakR_1(y)} \otimes s_1 \right)_{1_y},0, \cdots)](m) \right \}\\
    \frakF_2~is~a~copresh.=& \frakF_2[(\cdots, 0, \left( 1_{\frakR_1(y)} \otimes s_1' \right)_{1_y},0, \cdots) \circ (\cdots, 0, \left( 1_{\frakR_1(y)} \otimes s_1 \right)_{1_y},0, \cdots)] (m) \\
    =& \frakF_2[(\cdots, 0, \left( 1_{\frakR_1(y)} \otimes s_1 s_1' \right)_{1_y},0, \cdots)](m) \\
    rings~are~comm.=& \frakF_2[(\cdots, 0, \left( 1_{\frakR_1(y)} \otimes s_1' s_1 \right)_{1_y},0, \cdots)](m) \\
    =& (s_1' s_1) \cdot m .\\
\end{align*}

One can check that the functors $\Phi$ and $\Psi$ give rise to the desired equivalence. This completes the proof.
\end{proof}

\section*{Acknowledgments}
I would like to thank my advisor Prof. Fei Xu \begin{CJK*}{UTF8}{}
\CJKtilde \CJKfamily{gbsn}(徐斐) \end{CJK*} in Shantou University. 
%\printbibliography
\bibliographystyle{plain}
\bibliography{ref}
\end{document}